\documentclass[11pt]{amsart}
\usepackage{amsmath, amssymb, amscd, amsthm, amsfonts}
\usepackage{graphicx}
\usepackage{stmaryrd}
\usepackage{mathtools}
\usepackage[hidelinks]{hyperref}
\usepackage{tikz-cd}

\title{Algebraic Mapping Class Group Rigidity}
\author{Seong Youn Kim}

\newtheorem{theorem}{Theorem}
\numberwithin{theorem}{section}

\newtheorem{proposition}[theorem]{Proposition}
\newtheorem{corollary}[theorem]{Corollary}

\newtheorem*{metaconjecture}{Metaconjecture}
\newtheorem*{maintheorem}{Main Theorem}
\newtheorem*{ntheorem}{Theorem}
\theoremstyle{definition}
\newtheorem{definition}[theorem]{Definition}
\theoremstyle{remark}
\newtheorem{example}[theorem]{Example}

\newtheorem*{remark}{Remark}

\newcommand{\A}{{\mathbb{A}}}

\newcommand{\C}{{\mathbb{C}}}

\renewcommand{\H}{{\mathbb{H}}}

\newcommand{\R}{{\mathbb{R}}}

\newcommand{\Z}{{\mathbb{Z}}}

\newcommand{\Ccal}{{\mathcal{C}}}

\newcommand{\Gcal}{{\mathcal{G}}}

\newcommand{\Ocal}{{\mathcal{O}}}

\newcommand{\Tcal}{{\mathcal{T}}}

\newcommand{\Vcal}{{\mathcal{V}}}

\renewcommand{\char}{\textup{char}}
\renewcommand{\c}{\textup{c}}

\newcommand{\id}{{\textup{id}}}

\DeclareMathOperator{\Aut}{Aut}

\DeclareMathOperator{\Diff}{Diff}

\DeclareMathOperator{\GL}{GL}

\DeclareMathOperator{\Hom}{Hom}

\DeclareMathOperator{\sheafhom}{\mathcal{H}\kern -.5pt \emph{om}}

\DeclareMathOperator{\img}{im}

\DeclareMathOperator{\Mod}{Mod}

\DeclareMathOperator{\Out}{\textup{Out}}

\DeclareMathOperator{\PGL}{PGL}
\DeclareMathOperator{\PSL}{PSL}

\DeclareMathOperator{\rk}{rk}
\DeclareMathOperator{\Rep}{Rep}

\DeclareMathOperator{\SL}{SL}
\DeclareMathOperator{\SO}{SO}

\DeclareMathOperator{\Spec}{Spec}

\DeclareMathOperator{\Supp}{Supp}

\DeclareMathOperator{\tr}{tr}

\begin{document}

\maketitle

\begin{abstract}
Let $g, n \geq 0$ and $\Sigma = \Sigma_{g, n}$ be a connected oriented surface of genus $g$ with $n$ punctures. The $\SL_2$-character variety of $\Sigma$ has a rigid relative automorphism group as a finite extension of the mapping class group. The exceptional isomorphism between the $\SL(2, \C)$-character variety and the moduli of points on complex $3$-sphere provides a new description of the mapping class group of certain $\Sigma$. 
\end{abstract}

\section{Introduction}

In \cite{ivanov2006fifteen}, Ivanov formulated the metaconjecture for the mapping class group of surfaces.

\begin{metaconjecture}
    Every object naturally associated with a surface $\Sigma$ and having a sufficiently rich structure has $\Mod(\Sigma)$, the mapping class group of $\Sigma$, as its automorphism group. Moreover, this can be proved by a reduction to the theorem about the automorphism of the curve complex.
\end{metaconjecture}

The \textit{mapping class group rigidity} problem concerns identifying objects for which the conjecture holds. The conjecture for the curve complex itself was proven by Ivanov for closed surfaces, and later generalized by Korkmaz and Luo \cite{ivanov1997automorphisms, korkmaz1999automorphisms, luo1999automorphisms}.

The space of \textit{measured laminations}, denoted by $\mathcal{ML}(\Sigma)$ endowed with the intersection form, is also a natural object for the metaconjecture. Luo \cite{luo1999automorphisms} first proved that the automorphism of $\mathcal{ML}(\Sigma)$ induces automorphisms of the curve complex as desired, which is also observed by \cite{ohshika2018homeomorphismes, jyothis2023towards}.

Julien March\'e and Christopher-Lloyd Simon proved the conjecture for $\SL(2, \C)$-character variety of closed surfaces \cite{marche2021automorphisms}. We refer to this phenomenon as \textit{algebraic mapping class group rigidity}. This paper proves the relative version of the result.

To formulate the problem appropriately, let us observe the mapping class group of surfaces with punctures or boundaries naturally acts on the \textit{relative character varieties} regardless of the monodromy along the punctures or boundaries. However, one cannot expect the automorphism groups of all fibers to be isomorphic since the fibers themselves are not mutually isomorphic. Serge Cantat observed this for $\SL(2, \C)$-character variety for $4$-punctured sphere \cite{cantat2007holomorphic}.

On the other hand, as Cantat's result \cite{cantat2007holomorphic} shows, the automorphism group of the generic fiber consists only of mapping class group symmetries. Let us define the group as \textit{relative automorphism group} of the $\SL_2$-character variety.

\begin{maintheorem}
    Let $(g, n) \neq (1, 2)$ and $\Sigma = \Sigma_{g, n}$ be a surface of genus $g$ having $n$ punctures. The relative automorphism group of the $\SL_2$-character variety of $\Sigma$ is generated by the image of the mapping class group and the central representations fixing the boundary monodromies.
\end{maintheorem}

This recovers the result of March\'e-Simon \cite{marche2021automorphisms} and is also compatible with the result of Cantat \cite{cantat2007holomorphic} for four-punctured sphere and once-punctured torus. Also we will see the proof works regardless of the choice of the base domain if it is of characteristic zero.

Let us sketch the proof of the main theorem of \cite{marche2021automorphisms}, focusing on how the problem is reduced to the curve complex. 

Let $X = \Spec A$ be an affine integral scheme. Then $\Aut(X)$ fixes the generic point of $X$ and the \textit{Berkovich analytic boundary} of $X$ at the generic point, the rank $1$ valuations on $K(A)$ that do not contain $A$ in its valuation ring. 

March\'e-Simon constructed an embedding of $\mathcal{ML}(\Sigma)$ into the Berkovich analytic boundary of the $\SL(2, \C)$-character variety of closed surfaces. Using the \textit{tropical Fuchsian domination}, which will be introduced, we conclude that $\mathcal{ML}(\Sigma)$ is invariant under the action of the automorphism group.

\subsection*{Fuchsian domination}

Let $\rho : \pi_1 \Sigma \rightarrow \SL(2, \C)$ be a representation. For any $\gamma \in \pi_1 \Sigma$, consider the \textit{length spectrum functional}
\[
\ell_\rho(\gamma) = \inf_{x \in \H^3} d_{\H^3}(x, \rho(\gamma) x).
\]
Gu\'eritaud-Kassel-Wolff showed that a non-Fuchsian closed surface group representation into $\SL(2, \R)$ is dominated by a Fuchsian one and Deroin-Tholozan proved a similar result for representations into $\SL(2, \C)$ \cite{gueritaud2015compact, deroin2016dominating}.

\begin{ntheorem}[Gupta-Su \cite{gupta2023dominating}] Let $\Sigma = \Sigma_{g, n}$ and $\rho : \pi_1 \Sigma \rightarrow \SL(2, \C)$. There exists a Fuchsian representation $\rho_0$ dominating $\rho$ in the sense that for any $\gamma \in \pi_1 \Sigma$,
\[
\ell_\rho(\gamma) \leq \ell_{\rho_0}(\gamma)
\]
and the equality holds for any peripheral curve.
    
\end{ntheorem}

Consider the \textit{Zariski-Riemann compactification} of the character variety \cite{morgan1984valuations, otal2014compactification}. Consider the $\SL(2, \C)$-character variety of a surface, a rank $1$ valuation $v$ and a valuating sequence $[\rho_n]$ corresponding to $v$. As far as the author knows, it is difficult to determine whether $(\rho_n)_0$, the sequence of dominating Fuchsian representations is itself a valuating sequence or defines a rank $1$ valuation. Nevertheless the analogue for valuations, which we call \textit{tropical Fuchsian domination} holds.

\begin{ntheorem}[Tropical Fuchsian domination, March\'e-Simon \cite{marche2021automorphisms}]
    Let $\Sigma = \Sigma_{g, n}$ and $\Spec A$ be the $\SL(2, \C)$-character variety of $\Sigma$. Let $v$ be a rank $1$ valuation defined on $A$. There exists a measured lamination $\lambda$ and a rank $1$ valuation $v_\lambda$  such that for any $\gamma \in \pi_1 \Sigma$, 
    \[
    v(\tr_\gamma) \leq v_\lambda (\tr_\gamma)
    \]
    and the equality holds if $\gamma$ is a simple closed curve. Here we consider the valuation with a sign convention opposite to the usual one.
\end{ntheorem}

\subsection*{Construction of $\Mod(\Sigma_{g,1})$ and $\Mod(\Sigma_{g, 2})$} Yu-Wei Fan and Junho Peter Whang constructed an exceptional isomorphism between the $\SL(2, \C)$-character variety of a surface with $1$ or $2$ boundary components and the moduli of points on complex sphere $S^3$. Notably the construction of the latter can be carried out without \textit{a priori} reference to a surface.

As we will explicitly describe the structure of the automorphism group, one can also define $\Mod(\Sigma_{g,1})$ and $\Mod(\Sigma_{g,2})$ without the \textit{a priori} notion of the surface, unless $(g, n) = (1, 2)$.

Recall the Dehn-Nielsen-Baer theorem identifies the mapping class group of surfaces with the outer automorphism group of a fundamental group of surfaces preserving the boundary conjugacy classes. Even in comparison to the formulation, which uses the idea of the fundamental group of surfaces, the moduli of points on sphere only concerns the Coxeter invariant given by $\mathrm{Pin}(4)$.

\subsection*{Organization of the paper}
Section \ref{section2} introduces the necessary setup. We will recall some $\SL_2$-character variety theory based on Przytycki-Sikora \cite{przytycki1997skein, przytycki2019skein} and define the relative automorphisms. We will adapt the notion of valuation at our setting. We will recall some real tree theory based on Morgan-Shalen \cite{morgan1984valuations} and some measured lamination theory.

In Section \ref{section3} we carry out the construction of March\'e-Simon \cite{marche2021automorphisms} in the slightly modified setting. The author also modify some arguments of \cite{marche2021automorphisms}, for example proposition \ref{DimIntersection}. Finally, using the known cases of lower complexity surface character variety, we prove the main theorem.

In Section \ref{section4}, we will explicitly construct the exceptional isomorphism between the $\SL(2, \C)$-character variety and the moduli of points on sphere. We will also construct $\Mod(\Sigma_{g,1})$ and $\Mod(\Sigma_{g,2})$ as a quotient of the automorphism group of the moduli of points on sphere.

\subsection*{Acknowledgements}
I am deeply grateful to Junho Peter Whang for introducing the problem and his result which led to Corollary \ref{mcg}. This work was partially supported by the Samsung Science and Technology Foundation under Project Number SSTF-BA2201-03.

\section{Ingredients}\label{section2}
\subsection{Character varieties}

Let $g, n \geq 0$. From now on, $\Sigma = \Sigma_{g,n}$ always denotes a connected oriented surface of genus $g$ with $n$ punctures with $\chi(\Sigma) < 0$.  Denote $\pi = \pi_1 \Sigma$. Let $R$ always denote a commutative ring.

The $\SL_2$-\textit{character variety of $\Sigma = \Sigma_{g,n}$ over $R$} is the GIT quotient scheme
\[
    \Rep(\pi, \SL_2) \sslash \SL_2
\]
where $\SL(2, R)$ acts on $\Rep(\pi, \SL_2)$ by conjugation. There is a natural model over $\Z$ namely the \textit{$\SL_2$-skein character algebra}
\[
X(\Sigma) = \Spec \Z[\tr_\gamma]_{\gamma \in \pi} / (\tr_{1} - 2, \tr_\alpha\tr_\beta - \tr_{\alpha \beta} - \tr_{\alpha \beta^{-1}}).
\]
One can consider $\tr_\gamma$ as a regular function on the character variety over $R$ given by the character of the representation, $\rho \mapsto \tr \rho(\gamma)$. If $k$ is an algebraically closed field of characteristic zero, taking $k$-points of the variety recovers the closed conjugation orbits of $\SL(2, k)$-representations. $X(\Sigma)$ is an integral scheme \cite{przytycki2019skein}.

Note that $\tr_\gamma = \tr_1 \tr_\gamma - \tr_\gamma = \tr_{\gamma^{-1}}$ so one does not have to orient the immersed (smooth) curve realizing $\gamma$. $\tr_{\gamma \delta} = \tr_\gamma \tr_\delta - \tr_{\gamma \delta^{-1}} = \tr_\delta \tr_\gamma - \tr_{\delta \gamma^{-1}} = \tr_{\delta \gamma}$ so it suffices to consider the free homotopy classes of a curve realizing $\gamma$. 

Denote the function field of $X(\Sigma)$ as $K$. There exists a \textit{tautological representation} $\rho_{\text{taut}} : \pi \rightarrow \SL(2, L)$ where $L/K$ is a finite extension and 
\[
    \tr \rho_{\text{taut}}(\gamma) = \tr_{\gamma}
\]
for all $\gamma \in \pi$ \cite{marche2015character}. \ \\

\subsection{Relative automorphisms}

\begin{definition}
    Let $\Sigma = \Sigma_{g, n}$ and $\{c_i\}_{i=1}^n \subseteq \pi$ be $n$ curves each homotopic to a puncture. Let $A_\Sigma =  \Z[\tr_\gamma]_{\gamma \in \pi} / (\tr_{1} - 2, \tr_\alpha\tr_\beta - \tr_{\alpha \beta} - \tr_{\alpha \beta^{-1}})$. The \textit{relative automorphism group} of $X(\Sigma)$ is
    \begin{align*}
        \Aut^*(X(\Sigma)) =& \{ \phi \in \Aut_{\mathbf{Ring}}(A_\Sigma) : \phi(\tr_{c_i}) = \tr_{c_i} \} \\
        =& \Aut_{\Z[\tr_{c_i}]_{i=1}^n\text{-}\mathbf{Alg}}(A_\Sigma).
    \end{align*}
\end{definition}

\begin{remark}
    For any domain $S$, the proof will work for the base change by $S$ and then
    \[
    \Aut^*(X(\Sigma)) = \Aut^*(X(\Sigma)_S) = \Aut_{S[\tr_{c_i}]_{i=1}^n\text{-}\mathbf{Alg}}(S \otimes A_\Sigma).
    \]
    In particular, we may let $S = \C$ and consider relative automorphisms of the $\SL(2, \C)$-character variety as an analytic variety. We write $\C[X(\Sigma)] = \C \otimes A_\Sigma$.
\end{remark}

A \textit{multicurve} is a finite collection of disjoint simple closed curves. A multicurve is \textit{reduced} if it does not contain any contractible components. If $\Gamma = \sqcup_i \gamma_i$, we write $\tr_\Gamma = \prod_i \tr_{\gamma_i}$. The canonical basis of $A_\Sigma$ is well known.

\begin{theorem}[Przytycki-Sikora, \cite{przytycki1997skein}]
    $A_\Sigma$ is a free abelian group and $\{ \tr_\Gamma \}$, where $\Gamma$ runs over the reduced multicurves on $\Sigma$, forms a basis of $A_\Sigma$.
\end{theorem}

A simple closed curve on $\Sigma$ is called \textit{essential} if it is not nullhomotopic and not homotopic to the punctures. As a restatement of the above, $\{ \tr_\Gamma \}$, where $\Gamma$ runs over the \textit{essential} multicurves on $\Sigma$, forms a basis of $A_\Sigma$ as a $\Z[\tr_{c_i}]_{i=1}^n$-module . The expansion of an element of the character algebra with respect to the (essential) multicurve basis will be called the \textit{(essential) multicurve decomposition}.

$X(\Sigma)$ has a natural action of $\Out(\pi)$ by precomposition. Write $\{ x_i \}_{i=1}^n$ as punctures of $\Sigma$. The Dehn-Nielsen-Baer theorem says the (extended relative) \textit{mapping class group} of $\Sigma$
\[
\Mod(\Sigma) = \ker (\pi_0\Diff(\Sigma, \{x_i\}_{i=1}^n) \rightarrow S_n)
\]
embeds into $\Out(\pi)$. This yields a natural morphism $\Mod(\Sigma) \rightarrow \Aut^*(X(\Sigma))$.

\begin{remark}
    One can consider $\Sigma = \Sigma_{g,n}$ as a punctured surface or a compact surface with boundaries, not affecting $X(\Sigma)$. For surfaces with nonempty boundaries, the natural morphism $\Mod(\Sigma) \rightarrow \Aut^*(X(\Sigma))$ has a nontrivial kernel containing boundary Dehn twists.
\end{remark}

\begin{example}
    Let $\Sigma = \Sigma_{1, 1}$ and write $\pi = F_2 = \langle a, b \rangle$. Then $A_\Sigma = \Z[x, y, z]$ where
    \[
    (x, y, z) = (\tr_a, \tr_b, \tr_{ab})
    \]
    and
    \[
    \tr_{c_1} = \tr_{[a,b]} = x^2 + y^2 + z^2 - xyz - 2.
    \]
    Consider three mapping classes generating $\Mod(\Sigma) = \GL(2, \Z)$
    \[
    \begin{pmatrix}
        0 & 1 \\
        1 & 0
    \end{pmatrix},
    \begin{pmatrix}
        1 & 1 \\
        0 & 1
    \end{pmatrix},
    \begin{pmatrix}
        1 & 0 \\
        0 & -1
    \end{pmatrix}
    \]
    realized by the automorphism of $F_2$ as
    \begin{align*}
        (a, b) &\mapsto (b, a) \\ 
        (a, b) &\mapsto (a, ab) \\
        (a, b) &\mapsto (a, b^{-1}).
    \end{align*}
    
    This yields a homomorphism $\GL(2, \Z) \rightarrow \Aut^*(X(\Sigma))$ whose image is generated by three transformations
    \begin{align*}
        (x, y, z) &\mapsto (y, x, z) \\
        (x, y, z) &\mapsto (x, xy-z, y) \\
        (x, y, z) &\mapsto (x, y, xy-z).
    \end{align*}
    This also is equal to the group generated by a single Vieta involution and $S_3$ permuting coordinates. Its kernel is generated by
    \[
    \begin{pmatrix}
        -1 & 0 \\
        0 & -1
    \end{pmatrix}
    \]
    which corresponds to the hyperelliptic involution.

\end{example}

\subsection{The curve complex}

\begin{definition}
    The \textit{curve complex} $\Ccal(\Sigma)$ is a simplicial complex whose vertices are the free homotopy classes of essential simple closed curves. They form a simplex if and only if they are mutually disjoint.
\end{definition}

Here is the rigidity theorem of the curve complex.

\begin{theorem}[Luo, \cite{luo1999automorphisms}]
    Let $3g + n \geq 5$. The group homomorphism $\Mod(\Sigma) \rightarrow \Aut(\Ccal(\Sigma))$ is surjective unless $(g, n) = (1, 2)$ and injective unless $(g, n) = (1, 2), (2, 0)$. The kernel is generated by hyperelliptic mapping classes. If $(g, n) = (1, 2)$, the image of the morphism is an index $5$ subgroup of $\Aut(\Ccal(\Sigma))$.
\end{theorem}

Let $\Aut(\Ccal(\Sigma))$ be the automorphism group of $\Ccal(\Sigma)$ as simplicial complex. By the result above, the morphism $\Mod(\Sigma) \rightarrow \Aut^*(X(\Sigma))$ factors as
\begin{align}\label{mod to aut}
    \Mod(\Sigma) \rightarrow \Aut(\Ccal(\Sigma)) \hookrightarrow \Aut^*(X(\Sigma))
\end{align}

Here the second arrow is injective, since given for any $\varphi \in \Aut(\Ccal(\Sigma))$, the corresponding relative automorphism is given by $\tr_\gamma \mapsto \tr_{\phi(\gamma)}$ for any class of a simple closed $\gamma \in \pi$.

\subsection{Valuations}\label{section2.2}

\begin{definition}\label{valdef}
    Let $R$ be a domain and $A$ be an $R$-algebra as well as a domain. Let $\Lambda$ be a submonoid of $(\R_{\geq 0}, +)$. A \textit{(boundary rank 1) valuation} on $A$ is a function $v : A \rightarrow \{ -\infty \} \cup \Lambda $ satisfying
    \begin{enumerate}
        \item $v(0) = -\infty$, $v(r) = 0$ if $r \in R \setminus \{ 0 \}$.
        \item $v(a) > 0$ for some $a \in A$.
        \item $v(fg) = v(f) + v(g)$ for all$ f, g \in A$.
        \item $v(f + g) \leq \max \{ v(f), v(g) \}$, with equality if $v(f) \neq v(g)$ for all $f, g \in A$.
    \end{enumerate}
\end{definition}

The valuation $v$ naturally extends to a valuation on its field of fractions $K(A)$ by letting $v(f/g) = v(f) - v(g)$, where $\Lambda$ extends to its group completion naturally embedded in $\R$. Then $A$ is \textit{not contained} in the \textit{valuation ring}
\[
\Ocal_v = \{ x \in K(A) : v(x) \leq 0 \}
\]
of $K(A)$ unless $A$ is a field. Note that a valuation can be defined on a field $K$ without specifying a domain $A \subseteq K$. Denote the space of valuations on an $R$-algebra $A$ by $\Vcal_R(A)$. $\Vcal_R(A)$ has a natural pointwise convergence topology and the $\R_{>0}$-action given by changing the monoid embedding.

\begin{remark}
     Let $R, A$ be as in Definition \ref{valdef}. The \textit{Berkovich analytic compactification} of $\Spec A$ over $R$ is the collection of rank $1$ valuations on $\kappa(s)$ extending the valuation on $R$ where $s \in \Spec A$ and $\kappa(s)$ denotes the residue field at $s$. As the action of $\Aut_{R\text{-}\mathbf{Alg}}(A)$ on $\Spec A$ preserves its generic point, the action also preserves the Berkovich boundary points at the generic point, especially extending the trivial valuations along $R$. Thus $\Aut_{R\text{-}\mathbf{Alg}}(A)$ acts on $\Vcal_R(A)$, which is explicitly written as
     \[
     (\varphi \cdot v)(a) = v(\varphi^{-1} a)
     \]
     where $\varphi \in \Aut_{R\text{-}\mathbf{Alg}}(A)$, $v \in \Vcal_R(A)$ and $a \in A$. This is also a continuous action.
\end{remark}

The multicurves are themselves forming a basis of the character algebra $A_\Sigma$ as well as an integral \textit{geodesic lamination} compactifying the Teichm\"uller space. We will construct the pairing
\[
(\tr_\gamma, \tr_\delta) = v_{\gamma}(\tr_\delta) = i(\gamma, \delta)
\]
by mimicking the duality map of Fock-Goncharov. Reflecting this, the valuation defined here is different from its sign with the traditional one. 

\subsection{Real buildings}

\begin{definition}[Morgan-Shalen, \cite{morgan1984valuations}]
    Let $K$ be a field with a valuation $v$. The \textit{Bruhat-Tits real tree} with respect to $\GL(2, K)$ is a complete metric space $T$ defined as follows:
    \begin{enumerate}
        \item There are subsets called \textit{vertices} of $T$
        \[
        V(T) = \{ \text{Homothety classes of complete $\Ocal_v$-lattices in $K^2$} \}.
        \]
        \item The metric of $T$ between vertices is given as follows : for each two vertices $[L_1], [L_2]$ represented by lattices $L_1, L_2$, there exists $a \in K$ and $x \in \Ocal_v$ such that $aL_1 \leq L_2$ and $L_2 / aL_1 \simeq \Ocal_v / x\Ocal_v$. Then the distance $d_T$ between the vertices is defined as
        \[
        d_T([L_1], [L_2]) = v(x).
        \]
        In particular, the $GL(2, K)$-action on $V(T)$ is transitive.
        \item For any two vertices $v_1, v_2$ of $T$ with $d_T(v_1, v_2) = d$, there exists an isometric embedding $[0, d] \rightarrow T$ called \textit{arcs}. This produces a subset $T_v \subseteq T$ consisting of vertices and arcs.
        \item $T$ is the metric completion of $T_v$. 
    \end{enumerate}
\end{definition}

This is the simplest type of a \textit{real} Euclidean building. We call a valuation \textit{discrete} if the value group $\img v$ is discrete. If $v$ were discrete, the constructed metric tree is just a metric rescaling of the \textit{Bruhat-Tits tree}.

Note that $\GL(2, K)$ acts by isometries of the vertices of $T$, which extends uniquely to an isometry of $T$. Furthermore, the \textit{translation length} of an isometry $\gamma \in \SL(2, K) \leq \GL(2, K)$ is completely determined by its trace,
\[
\inf_{x \in T}(x, \gamma \cdot x) = \max\{0, 2 v(\tr \gamma)\}.
\]

Denote $\Vcal = \Vcal_{\Z[\tr_{c_i}]_{i=1}^n}(A_\Sigma) $. If one choose a $v \in \Vcal$, then the tautological representation $\rho_{\text{taut}} : \pi \rightarrow \SL(2, L)$ gives an action of $\pi$ to the Bruhat-Tits real tree with respect to $\hat{v}$, extending $v$ to $L$. Then the translation length of $\gamma \in \pi$ is given as
\[
\inf_{x \in T}(x, \gamma \cdot x) = \max\{0, 2 \hat{v}(\tr \rho_{\text{taut}}(\gamma)) = 2 v (\tr_\gamma) \}.
\]

\subsection{Measured geodesic laminations}

In this subsection, we restrict ourselves to the case that $\Sigma = \Sigma_{g, n}$ is a surface endowed with a finite-volume hyperbolic structure. We will recall the notion of measured geodesic laminations of a punctured surface, via geodesic currents \cite{bonahon1988geometry, lindenstrauss2008ergodic}. 

Write $\Gamma \leq \PSL(2, \R)$ as the holonomy of $\Sigma$ and let $S \subseteq \partial \H^2$ be its nonparabolic points. Note $\Gcal(\Sigma) = (S \times S \setminus \Delta) / (\Z/2\Z)$ parametrizes complete unoriented geodesics on $\H^2$ whose projection to $\Sigma$ is compactly supported. A $\Gamma$-invariant locally finite positive Borel measure on $\Gcal(\Sigma)$ is called a \textit{geodesic current} on $\Sigma$. 

The current space has a natural weak-$*$ topology and a continuous bilinear form $i$ extending the geometric intersection number, which is called the \textit{intersection form}. The intersection form corresponds a weighted multicurve to a current via the pairing $\Gamma \mapsto i(\lambda,\Gamma)$.

A \textit{measured geodesic lamination} is a geodesic current whose self-intersection vanishes. For any measured geodesic lamination, there exists a corresponding unique \textit{transverse measure} on a \textit{geodesic lamination}. The space of measured geodesic laminations is denoted by $\mathcal{ML}(\Sigma)$. From now on we call a measured geodesic lamination by a \textit{measured lamination}.

$\mathcal{ML}(\Sigma)$ has a natural \textit{train track} coordinate system. This gives a piecewise linear structure on $\mathcal{ML}(\Sigma)$ and there exists a Lebesgue measure called the \textit{Thurston measure}. 

Weighted essential simple closed curves are dense in $\mathcal{ML}(\Sigma)$. Hence every intersection with a measured lamination can be comprehended as a limit of taking intersection with weighted essential simple closed curves.

A measured lamination $\lambda$ is \textit{filling} if for any essential simple closed curve $\gamma$, $i(\lambda, \gamma) > 0$. Almost all measured laminations are filling.

A measured lamination $\lambda$ is \textit{uniquely ergodic} if there exists a measured lamination $\lambda'$ with $\Supp \lambda = \Supp \lambda'$ then $\lambda = C \lambda'$ for some $C > 0$. Almost all measured laminations are uniquely ergodic \cite{masur1982interval}.

\section{March\'e--Simon Construction}\label{section3}

\subsection{Measured laminations as valuations}

The Teichm\"uller space $\Tcal(\Sigma)$ naturally embeds into the real points of $X(\Sigma)$. Furthermore one can construct an embedding between the compactification of the two spaces. Here we will deal with the unprojectivized versions, $\mathcal{ML}(\Sigma)$ and $\Vcal$ in order to compare the action of $\Aut^*(X(\Sigma))$ explicitly.

\begin{proposition}[March\'e-Simon, \cite{marche2021automorphisms}]
    Let $\lambda \in \mathcal{ML}(\Sigma)$ be a measured lamination. Then there exists a corresponding valuation defined by
    \[
    v_{\lambda}(\tr_{\Gamma}) = i(\lambda, \Gamma) = \sum_j i(\lambda, \gamma_j)
    \]
    and
    \[
    v_\lambda\left(\sum_{\Gamma} c_{\Gamma} \tr_{\Gamma}\right) = \max \{ v_\lambda(\tr_{\Gamma}) : c_{\Gamma} \neq 0 \}
    \]
    where $\Gamma = \bigsqcup_j \gamma_j$ is an essential multicurve on $\Sigma$.
\end{proposition}

The crucial part of the proof is to verify the consistency
\[
v_{\lambda}(\tr_\gamma) = i(\lambda, \gamma)
\]
for \textit{possibly nonsimple} $\gamma \in \pi$, which is a reformulation of the following theorem.

\begin{theorem}[D. Thurston, \cite{thurston2008geometric}]
    Let $\gamma \in \pi$ and $\tr_\gamma = \sum_{\Gamma} c_{\Gamma} \tr_{\Gamma}$ be its essential multicurve decomposition. Then for every simple closed curve $\delta$ in $\Sigma$,
    \[
    i(\delta, \gamma) = \max \{ i(\delta, \Gamma) : c_{\Gamma} \neq 0 \}.
    \]
\end{theorem}

Hence one gets a $\Mod(\Sigma)$-equivariant embedding $\mathcal{ML}(\Sigma) \rightarrow \Vcal$. 

\begin{definition}
    A valuation $v \in \Vcal$ is called \textit{dominating} if for any $f = \sum_{\Gamma} c_{\Gamma}\tr_{\Gamma} \in A_\Sigma$ essential multicurve decomposition,
        \[
        v\left(\sum_{\Gamma} c_{\Gamma} \tr_{\Gamma}\right) = \max \{ v(\tr_{\Gamma}) : c_{\Gamma} \neq 0 \}.
        \]
    In particular, for any measured lamination $\lambda \in \mathcal{ML}(\Sigma)$, the corresponding valuation $v_\lambda$ is dominating.
\end{definition}

\subsection{Tropical Fuchsian domination}

Let $\lambda$ be a measured lamination on $\Sigma$. Consider its preimage on its universal cover $\widetilde{\Sigma}$. This yields a \textit{dual real tree} denoted by $T_\lambda$, obtained as the metric completion of the tree whose vertices are connected components of the complementary region $\widetilde{\Sigma} \setminus \widetilde{\lambda}$, with the arc length connecting two vertices given by the measure of transverse arcs whose endpoints lie in two components corresponding to the vertices. Note that the constructed real tree has a dense subtree generated by its vertices, as the Bruhat-Tits real tree does. $\pi$ acts on $T_\lambda$ by the holonomy.

Let $\pi$ act on a real tree $T$. The action is called \textit{small (stabilizer)} if for any geodesic segment $I \subseteq T$, its stabilizer is cyclic. The action is called \textit{type-preserving} if any peripheral element of $\pi$ fixes some point in $T$. By its construction the $\pi$-action on $T_\lambda$ is small, minimal and type-preserving.

\begin{theorem}[Skora \cite{skora1996splittings}]
    If a $\pi$-action on a real tree $T$ is small, minimal and type-preserving then there exists $\lambda \in \mathcal{ML}(\Sigma)$ such that $T \simeq T_\lambda$ is isometric and $\pi$-equivariant.
\end{theorem}

If one removes the smallness condition, there is still a morphism between real trees. It may fold at certain points, thereby decreasing distances between some points.

\begin{theorem}[Morgan-Otal \cite{morgan1993relative}]
    If a $\pi$-action on a real tree $T$ is minimal and type-preserving then there exists $\lambda \in \mathcal{ML}(\Sigma)$ and a morphism between real trees $\Phi : T_\lambda \rightarrow T$ which is $1$-Lipschitz and $\pi$-equivariant.
\end{theorem}

Now consider the $\pi$-action on a Bruhat-Tits real tree by its tautological representation. There exists a measured lamination $\lambda$ such that for any $\gamma \in \pi$, by $1$-Lipschitzness
\[
2v(\tr_\gamma) \leq \inf_{x \in T}d_T(x, \gamma \cdot x) \leq d_{T_\lambda}(x_0, \gamma \cdot x_0) = i(\lambda, \gamma) = v_\lambda(\tr_\gamma).
\]
Now for any $f = \sum_{\Gamma} c_{\Gamma}\tr_{\Gamma} \in A_\Sigma$, we have
\[
v(f) \leq \max \{ v(\tr_\Gamma) : c_{\Gamma} \neq 0 \} \leq \max \{ v_{\frac 12 \lambda}(\tr_\Gamma) : c_{\Gamma} \neq 0 \} = v_{\frac 12 \lambda}(f).
\]

\begin{corollary}[March\'e-Simon \cite{marche2021automorphisms}]
    For any $v \in \Vcal$, there exists a unique measured lamination $\lambda$ such that $v(f) \leq v_\lambda(f)$ for any $f \in A_\Sigma$ and $v(\tr_\gamma) = v_\lambda(\tr_\gamma)$ for any $\gamma \in \pi$. In particular, every dominating valuation arises from measured laminations and $\mathcal{ML}(\Sigma) \rightarrow \Vcal$ is a closed embedding.
\end{corollary}

\begin{corollary}[March\'e-Simon \cite{marche2021automorphisms}]
    The action of $\Aut^*(X(\Sigma))$ on $\Vcal$ preserves $\mathcal{ML}(\Sigma)$.
\end{corollary}

\begin{proof}
    This is due to the fact that $\Aut_{R\text{-}\mathbf{Alg}}(A)$ preserves \textit{untamable valuations} 
    \[
    \{ v \in \Vcal_R(A) : \text{ if } v(a) \leq w(a) \text{ for any } a \in A \text{ then } w = Cv \text{ for some } C>0 \}.
    \]
\end{proof}

\begin{remark}
    Let $R$ be a characteristic zero domain. One may replace $X(\Sigma)$ by its base change $X(\Sigma)_R$. For $A = A_\Sigma$, if $R$ were a domain then $R \otimes A$ is also a domain \cite{przytycki2019skein}. Denote $\Vcal^\mathrm{dom}_R(A)$ as the subspace of $\Vcal_R(A)$ consisting of dominating valuations. Then there exists a map
    \[
    \Vcal^\mathrm{dom}_\Z(A) \rightarrow \Vcal^\mathrm{dom}_R(R \otimes A)
    \]
of extending valuation $v \mapsto v_R$ by
    \[
    v_R\left( \sum_\Gamma (r_\Gamma \otimes c_\Gamma) \tr_\Gamma \right) := \max \{v(\tr_\Gamma) : r_\Gamma \otimes c_\Gamma \neq 0 \}
    \]
which is a topological embedding. Moreover, for any $w \in \Vcal_R(A)$ one has $w(r \otimes a) = w(1 \otimes a)$ so there exists a natural section $w \mapsto v$ by
\[
v(a) := w(1 \otimes a)
\]
and $\Vcal^\mathrm{dom}(A)$ is homeomorphic to $\Vcal^\mathrm{dom}_R(R \otimes A)$.
\end{remark}

\subsection{Reduction to curve complex}

From now on we abuse $\Sigma = \Sigma_{g,n}$ for a compact connected oriented surface of genus $g$ with $n$ boundaries to facilitate the cohomology notation. The above corollary reduces the action of $\Aut^*(X(\Sigma))$ to the simplicial action on the curve complex. 

\begin{proposition}[March\'e-Simon \cite{marche2021automorphisms}]
Let $\lambda \in \mathcal{ML}(\Sigma)$. The value group of $v_\lambda$ is a subgroup of $\Z$ if and only if $\lambda = \frac{1}{2}\Gamma$, where $\Gamma$ is an essential multicurve represented as trivial classes in $H_{1}(\Sigma; \Z / 2\Z) = H^1(\Sigma, \partial \Sigma ; \Z / 2\Z)$.
\end{proposition}

\begin{proposition}\label{krull}
The number of unparallel components of an essential multicurve $\Gamma$ is equal to
\[
\dim A_\Sigma -\dim \Ocal_{\Gamma}
\]
where $\dim$ denotes the Krull dimension and $\mathcal{O}_{\Gamma}:=\{f \in A_\Sigma : v_{\Gamma}(f) \leq 0\}$.
\end{proposition}

Observe
\[
\dim A_\Sigma - \dim \Ocal_\Gamma = \dim \C[X(\Sigma)] - \dim \C \otimes \Ocal_\Gamma
\]
and the proof of March\'e-Simon applies.

For a surface $\Sigma = \Sigma_1 \sqcup \cdots  \sqcup \Sigma_s$ with connected components $\Sigma_i$, consider $A_\Sigma = \bigotimes_{i=1}^s A_{\Sigma_i}$, which is also integral. Then $\Ocal_\Gamma$ is the image of the ring homomorphism $A_{\Sigma |_\Gamma} \rightarrow A_\Sigma$ induced by the inclusion, where $\Sigma |_\Gamma$ denotes the surface obtained by cutting $\Sigma$ along $\Gamma$.

\begin{proposition}\label{DimIntersection}
Let $\gamma, \delta$ be non-parallel essential simple closed curves. Then
\[
\dim \mathcal{O}_{\gamma} \cap \Ocal_{\delta} \leq \dim A_\Sigma - 2
\]
where the equality holds if and only if $i(\gamma,\delta)=0$ or $i(\gamma, \delta) = 1$ and $\Sigma = \Sigma_{1,1}$.
\end{proposition}

\begin{proof}
We may change the base to $\C$. Observe $\Ocal_\gamma \cap \Ocal_\delta$ corresponds to the image of $\C[X(\Sigma |_{(\gamma \cup \delta)})] \rightarrow \C[X(\Sigma)]$. Thus
\[
\dim \Ocal_\gamma \cap \Ocal_\delta \leq \dim \C[X(\Sigma |_{(\gamma \cup \delta)})]
\]
and it suffices to compute $\dim \C[X(\Sigma |_{(\gamma \cup \delta)})]$.

If $i(\gamma, \delta)=0$, the equality holds by the previous proposition.

If $i(\gamma, \delta)>0$, take the representatives of $\gamma, \delta$ with minimal intersection.
We will consider $\delta$ as arcs in the cut surface $\Sigma |_\gamma$. $\delta$ is divided into $i(\delta, \gamma)$ pieces of arcs connecting the boundary components. By cutting along each arc segment, the pair $(g, b)$ changes where $g$ denotes the genus and $b$ denotes the number of boundary components of the cut surface. 

By cutting along each arc, there are three possibilities. If the arc is separating, $(g, b) \rightarrow (g, b+1)$. This process will be denoted by \textit{pants cut}, a half cut of a (local) pants. If the arc is nonseparating, $(g, b) \rightarrow (g-1, b+1)$ or $(g, b) \rightarrow (g, b-1)$. These will be denoted by \textit{glued pants cut} and \textit{boundary merging}.

\begin{figure}[h]
    \centering
    \includegraphics[width=0.3\textwidth]{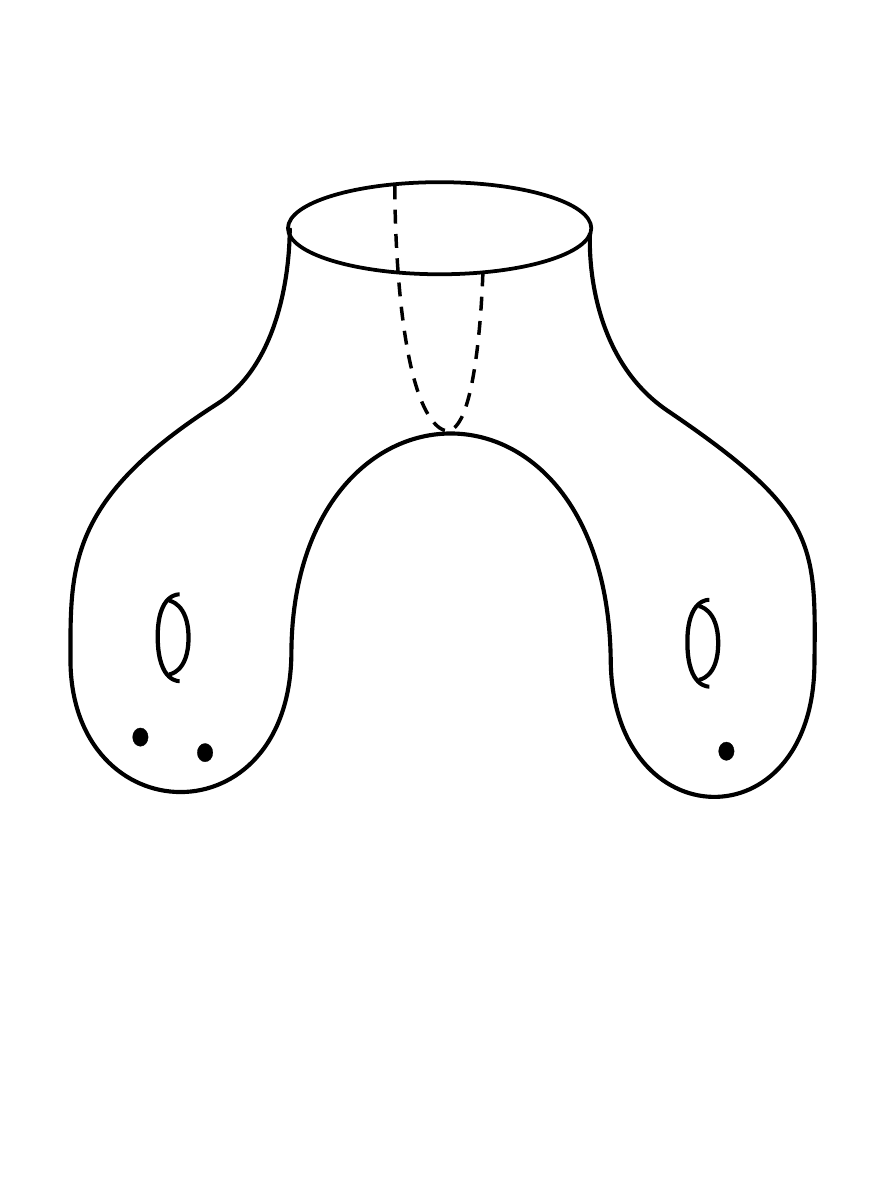}
    \includegraphics[width=0.3\textwidth]{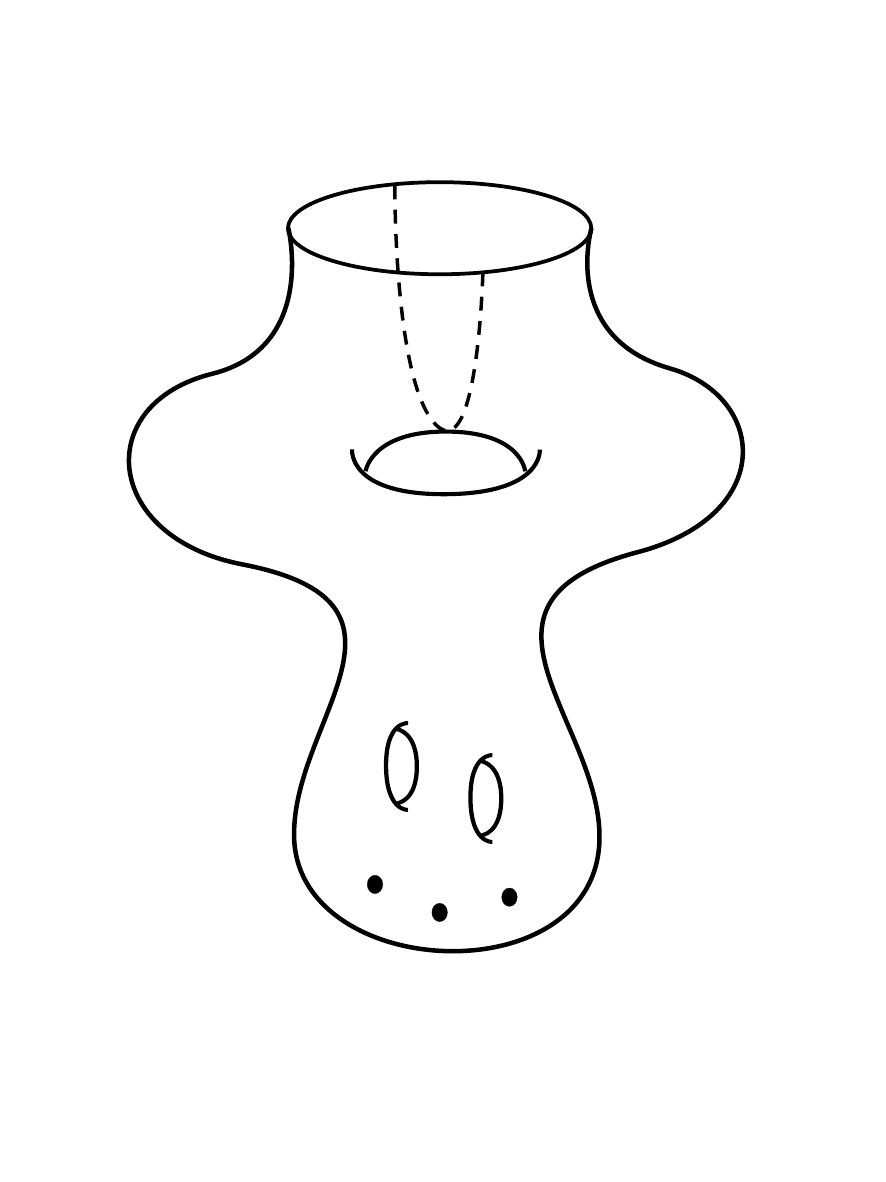}
    \includegraphics[width=0.3\textwidth]{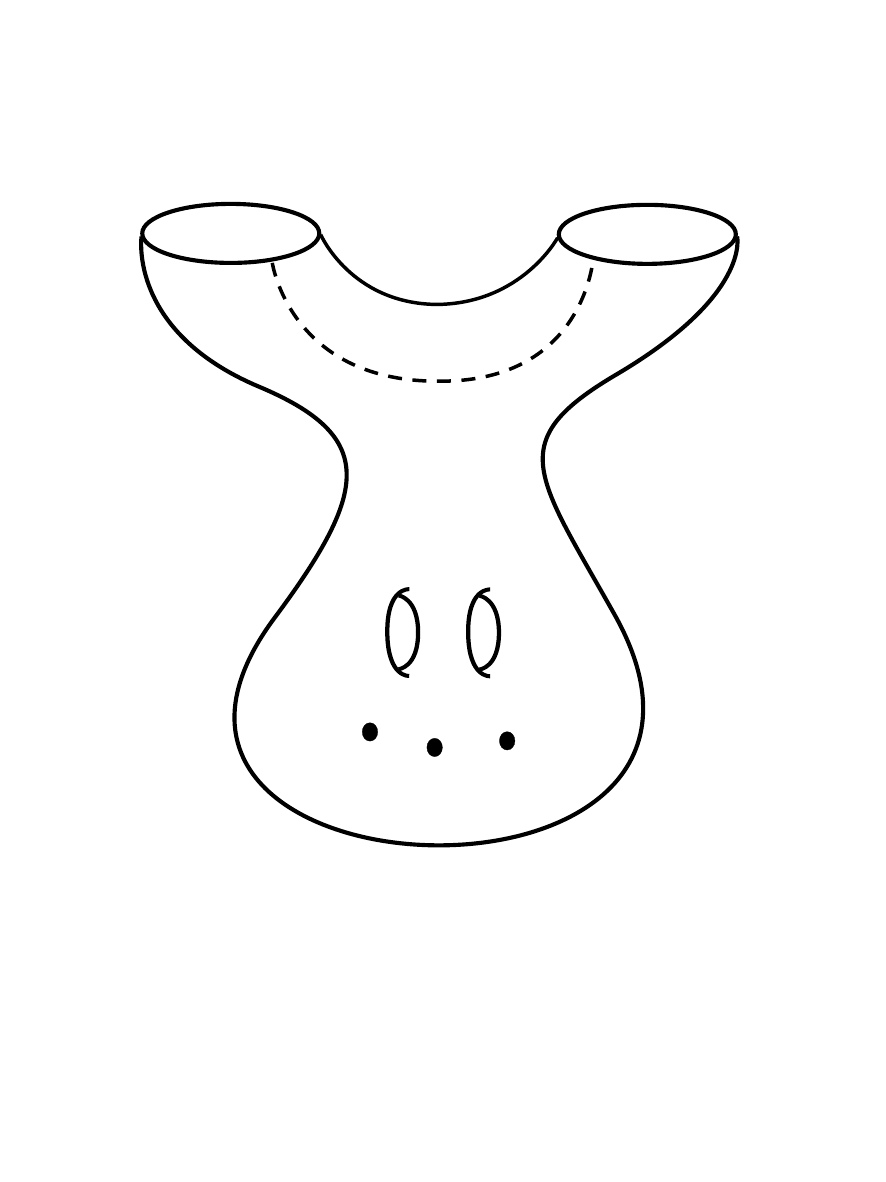}
    \caption{(i) Pants cut. (ii) Glued pants cut. (iii) Boundary merging. }
\end{figure}

In any case the Euler characteristic increases by 1. If the cut pieces do not contain any annuli or polygons, then the dimension decreases by $3$. 

Note that polygons are yielded only by boundary merging, and for this one needs a precedent annulus. Thus if one has a cut polygon then the dimension drops at least $3$.

A pants cut yields an annulus if and only if one of the leg is indeed a leg of pants, which is not local. Similarly, a glued pants cut yields an annulus if and only if it is a cut of a once-punctured torus. Finally, a boundary merging yields an annulus if and only if the component is a pants.

One concludes that if $\dim \Ocal_\gamma \cap \Ocal_\delta = \dim \C[X(\Sigma)] - 2$ then $i(\gamma, \delta) = 1$ and an annulus is yielded by a boundary merging. Thus $\Sigma = \Sigma_{1, 1}$.
\end{proof}

\begin{corollary}
    $\Aut^*(X(\Sigma))$ permutes the valuations along nonseparating simple closed curves and the valuations along $\frac 12$-weighted separating simple closed curves. In particular, it induces an action on $\Ccal(\Sigma)$.
\end{corollary}

\begin{proof}
    The value group of a valuation $v$ is isomorphic to $\Z$ if and only if $v = v_\lambda$ and $\lambda$ is a weighted multicurve consisting of $1$-weighted nonseparating simple closed curves and $\frac 12$-weighted separating simple closed curves.

    By Proposition \ref{krull}, $\Aut^*(X(\Sigma))$ induces an action on the $0$-skeleton of $\Ccal(\Sigma)$. For $3g + n < 5$, note $\Ccal(\Sigma)$ is a $0$-dimensional simplicial complex.

    By Proposition \ref{DimIntersection}, $\Aut^*(X(\Sigma))$ preserves the Krull dimension of the subalgebra $\Ocal_\gamma \cap \Ocal_\delta$. If $3g + n \geq 5$, the action on $0$-skeleton preserves tuples of vertices corresponding to a simplex of $\Ccal(\Sigma)$. Thus it induces an action on $\Ccal(\Sigma)$.
\end{proof}

\begin{corollary}
Let $3g + n \geq 5.$ If $(g,n) \neq (1,2), (2,0)$, there is a surjective group homomorphism 
\[
\Aut^*(X(\Sigma)) \twoheadrightarrow \Mod(\Sigma)
\] 
which is a section of \ref{mod to aut}. 
For $(g, n) = (2, 0)$, the result holds if one replaces $\Mod(\Sigma)$ by the quotient by its center.
\end{corollary}

\subsection{Central representations}

Let $\Sigma = \Sigma_{g, n}$ where $3g +n \geq 5$ and $(g, n) \neq (1, 2)$. 

Let $Z$ be the center of $\SL(2,R)$ over a domain $R$. The central representations $H^1(\Sigma, Z) = \Hom(\pi , Z)$ naturally act on the representation variety $\Rep(\pi, \SL_2)$ over $R$, given by considering
\[
\Rep(\pi, \SL_2) \rightarrow \Rep(\pi, \PSL_2)
\]
as a $H^1(\Sigma, Z)$-torsor. Since the representation is central, the action descends to the character variety.

Note that if $\char R = 2$ then $Z$ is trivial and $\char R \neq 2$ then $Z = \Z/2\Z$. If $Z$ is trivial then the following parts of the section are vacuously true and one may assume $Z = \Z / 2\Z$.

The action on $A_\Sigma$ is given by sign changes : $\tr_\gamma \leftrightarrow \chi(\gamma)\tr_\gamma$ for certain curves $\gamma \in \pi$ where $\chi \in H^1(\Sigma, \partial \Sigma ; \Z / 2\Z)$. Note that the cohomology $H^1(\Sigma, \partial \Sigma ; \Z / 2\Z) \leq H^1(\Sigma , \Z / 2\Z)$ fixes the monodromies along each puncture. 

If $R = \C$, by the exceptional isomorphism
\[
\begin{tikzcd}
{\mathrm{Spin}(3,1)} \arrow[d] \arrow[r, "\sim"] & {\SL(2, \C)} \arrow[d] \\
{\SO(3,1)^\circ} \arrow[r, "\sim"']                & {\PSL(2, \C)}         
\end{tikzcd}
\]
the action of central representations corresponds to the change of the spin structure of $\Sigma$.

Let $K$ be the kernel of $\Aut^*(X(\Sigma)) \twoheadrightarrow \Mod(\Sigma)$ or the replaced morphism in Corollary \ref{main} where $(g, n) = (2, 0)$. For $\varphi \in K$, for any simple closed curves $\gamma, \delta$ one has
\[
v_\delta(\varphi(\tr_\gamma)) = v_\delta(\tr_\gamma) = i(\gamma, \delta).
\]
Write the essential multicurve basis expansion of $\varphi(\tr_\gamma) = \sum_\Gamma c_\Gamma \tr_\Gamma$. The simplicity of $\gamma$ implies that if $i(\gamma, \delta) = 0$ and $c_\Gamma \neq 0$ then $i(\delta, \Gamma) = 0$. For any essential multicurve $\Gamma$, except in the case that $\Gamma$ consists only of parallel copies of $\gamma$, one can find $\delta$ such that $i(\delta, \Gamma) > 0$ and $i(\delta, \gamma) = 0$. Thus $\varphi $ induces an automorphism of the subalgebra generated by $\tr_\gamma$ and one can then write
\[
\varphi(\tr_\gamma) = a_\gamma \tr_\gamma + b_\gamma
\]
for $a_\gamma, b_\gamma \in \Z[\tr_{c_i}]_{i=1}^n$. Note that $\Z[\tr_{c_i}]_{i=1}^n$ is a subdomain of $A_\Sigma$.

\begin{proposition}[March\'e-Simon \cite{marche2021automorphisms}]
Any $\varphi \in K$ is induced by a central representation.
\end{proposition}

\begin{proof}
    The proof in March\'e-Simon works for any domain. It shows that $b_\gamma = 0$ and $a_\gamma = \pm 1$ for any $\gamma \in \pi$ simple and the assignment $\gamma \mapsto a_\gamma$ for simple curves is induced by a central representation.

    Consider $\gamma \delta \in \pi$ which is realized by a curve with a self-intersection. The multicurve expansion
    \[
    \tr_{\gamma\delta} = \tr_{\gamma}\tr_\delta - \tr_{\gamma \delta^{-1}}
    \]
    then $\gamma \cup \delta$ and $\gamma \delta^{-1}$ differ by an embedded rectangle near the self-intersection. Thus the two curves lie in the same homology class in $H_1(\Sigma, \Z / 2\Z)$. Inductively, all of the multicurves in the multicurve expansion of $\tr_\alpha$ where all $\alpha \in \pi$ represents the same homology class in $H_1(\Sigma, \Z / 2\Z)$.

    Thus every $\alpha \in \pi$ and $\varphi \in K$, $\varphi(\tr_\alpha) = a_\alpha \tr_\alpha$ where $a_\alpha = \pm 1$ and $a_\alpha a_\beta = a_{\alpha \beta}$ for any $\beta \in \pi$.
\end{proof}

\begin{corollary}\label{main}
    Let $\Sigma = \Sigma_{g, n}$ where $3g + n \geq 5$ and $(g, n) \neq (1, 2), (2, 0)$. There exists a split exact sequence
    \[
    1 \rightarrow H^1(\Sigma, \partial \Sigma ; \Z / 2\Z) \rightarrow \Aut^*(X(\Sigma)) \rightarrow \Mod(\Sigma) \rightarrow 1.
    \] 
    If $(g, n) = (2, 0)$ then one should replace $\Mod(\Sigma)$ by its center quotient to have the result.
\end{corollary}

\begin{remark}
    Let $R$ be a domain. One may change the base of $X(\Sigma)$ to $R$ to have the same result. If $\char R = 2$ then we have no central representations and $\Aut^*(X(\Sigma)_R) = \Mod(\Sigma)$.
\end{remark}

\subsection{Surfaces with low complexity}

In this subsection we will cover the exceptional surfaces with $(g, n) = (0, 3), (0, 4)$ and $(1, 1)$.

\subsubsection{Thrice-punctured sphere}

Let $\Sigma = \Sigma_{0,3}$. Then the short exact sequence in Corollary \ref{main} is trivial.

\subsubsection{Four-punctured sphere}

Let $\Sigma = \Sigma_{0,4}$. $X(\Sigma)(\C)$ is fibered by the \textit{relative character varieties} as complex varieties, each of which is a fiber of the morphism
\[
X(\Sigma)(\C) \rightarrow \C^4.
\]
Explicitly, it is given by the complex affine surface
\[
x^2 + y^2 + z^2 - xyz - Ax - By - Cz - D = 0
\]
where $A, B, C, D$ are determined by four boundary monodromy traces. For each fiber, the relative character variety is denoted by $X^{\textrm{rel}}(\Sigma)(\C)$. By the definition of $\Aut^*(X(\Sigma))$, for any fiber $X^{\textrm{rel}}(\Sigma)(\C)$ there exists an embedding
\[
\Aut^*(X(\Sigma))  \hookrightarrow \Aut(X^\textrm{rel}(\Sigma)(\C))
\]

\begin{remark}
To keep consistency with the result of Cantat-Loray, we chose the relative automorphism to fix each boundary trace. Alternatively, one can consider the automorphism of $A_\Sigma$ \textit{preserving} the subalgebra generated by boundary traces.
\end{remark}

\begin{theorem}[\`El huti, Cantat-Loray \cite{el1974cubic}\cite{cantat2007holomorphic}]
    For a generic choice of $X^{\mathrm{rel}}(\Sigma)(\C)$, its automorphism group is isomorphic to $\Mod(\Sigma)$. 
\end{theorem}

In this case every central representation is trivial; thus again Corollary \ref{main} holds.

\subsubsection{Once-punctured torus}

\begin{theorem}[Cantat-Loray, Perepechko \cite{cantat2007holomorphic}\cite{perepechko2021automorphisms}]
    Let $\Sigma = \Sigma_{1,1}$. $\Aut^*(X(\Sigma))$ is generated by a single Vieta involution
    \[
    \tau_z : (x, y, z) \leftrightarrow (x, y, xy-z),
    \]
    $S_3$ as permutations of $x, y, z$ and the double sign changes
    \[
    (x, y, z) \leftrightarrow (\varepsilon_1 x,  \varepsilon_2 y, \varepsilon_3 z)
    \]
    where $\varepsilon_i = \pm 1$ and $\varepsilon_1 \varepsilon_2 \varepsilon_3 = 1$.
\end{theorem}

Here the double sign changes correspond to the central representations and Corollary \ref{main} holds if one replaces $\Mod(\Sigma)$ with $\Mod(\Sigma) / \langle \iota \rangle = \PGL(2, \Z)$ where $\iota$ denotes the hyperelliptic involution $-\id$.

\begin{corollary}
    Let $\Sigma = \Sigma_{g,n}$ with $\chi(\Sigma) < 0$. Write $\Mod^*(\Sigma)$ as its image into $\Aut^*(X(\Sigma))$, which is isomorphic to the hyperelliptic involution quotient if $(g, n) = (1, 1)$ or $(2, 0)$. If $(g, n) \neq (1, 2)$, there exists a split exact sequence
    \[
    1 \rightarrow H^1(\Sigma, \partial \Sigma ; \Z / 2\Z) \rightarrow \Aut^*(X(\Sigma)) \rightarrow \Mod^*(\Sigma) \rightarrow 1.
    \]
\end{corollary}

Luo also showed that if $(g, n) \neq (1, 2)$, any automorphism of the curve complex should preserve the simplices corresponding to separating simple closed curves. That any element of $\Aut^*(X(\Sigma))$ fixes any trace along separating simple closed curve, seems not to be \textit{a priori} clear without the aid of the curve complex. 

\section{Moduli of points on spheres}\label{section4}

Let $S(4)$ be the complex hypersurface in $\A^4(\C)$ corresponding to a quaternary quadratic form, corresponding to a standard bilinear form $\langle -, - \rangle$. $S(4)$ naturally embeds into the \textit{Clifford algebra} corresponding to the quadratic space $(\A^4, \langle-, - \rangle)$. We call the closed algebraic subgroup generated by $S(4)$ the \textit{Pin group} denoted by $\mathrm{Pin}(4)$.

\begin{definition}
    The \textit{moduli of $r$ points on a $3$-sphere} is
    \[
    S(4, r) = S(4)^r \sslash \SO(4).
    \]
    The \textit{Coxeter invariant} is the morphism
    \[
    c : S(4, r) \rightarrow \mathrm{Pin}(4) \sslash \SO(4)
    \]
    given by $[(u_1, \ldots, u_r)] \mapsto [u_1 \otimes \cdots \otimes u_r]$.
\end{definition}

The \textit{braid group} over $r$ strands, denoted by $B_r$, is generated by $\sigma_1, \ldots, \sigma_{r-1}$ with relations
\begin{enumerate}
    \item $\sigma_i \sigma_j = \sigma_j \sigma_i$ if $|i - j| \geq 2$.
    \item $\sigma_i \sigma_j \sigma_i = \sigma_j \sigma_i \sigma_j$ if $|i - j| = 1$.
\end{enumerate}
$B_r$ acts on $S(4, r)$ via
\[
\sigma_i [(u_1, \ldots, u_r)] = [(u_1, \ldots, u_{i-1}, s_{u_i}(u_{i+1}), u_i, \ldots, u_r)]
\]
where $s_u(v) = 2\langle u, v \rangle u - v$ is a reflection along $u$.

The exceptional \textit{group} structure of $S^3$ makes it possible to construct an exceptional isomorphism.

\begin{theorem}[Fan-Whang \cite{fan2020surfaces}]
    Let $\Sigma = \Sigma_{g, n}$ with $n = 1, 2$ and $r = 2g + n$. There exists a $B_r$-equivariant isomorphism of complex algebraic varieties
    \[
    S(4, r) \simeq X(\Sigma)(\C) := \SL(2, \C)^{r-1} \sslash \SL(2, \C)
    \]
    where the $B_r$-action on the right hand side is given by the embedding $B_r \rightarrow \Mod(\Sigma)$ given by specifying hyperelliptic generators. Furthermore, the Coxeter invariant corresponds to the boundary monodromies along punctures of $\Sigma$.
\end{theorem}

The explicit isomorphism above goes as follows : Fix a system of hyperelliptic generators $\alpha_1, \ldots, \alpha_{r-1}$ of $\pi$ and a local system $[\rho]$, whose monodromies $\rho(\alpha_i)$ is denoted as $A_i \in \SL(2, \C)$. The corresponding points on sphere is given as
\[
[\left( A_1 \cdots A_{r-1}, A_2 \cdots A_{r-1}, \cdots, A_{r-1}, 1 \right)].
\]
The boundary monodromy is given as
\[
[(A_1 A_3 \cdots A_{r-2}) (A_1 A_2 \cdots A_{r-1})^{-1} (A_2 A_4 \cdots A_{r-1})] \quad \text{if } n=1
\]
and
\[
[(A_1 A_3 \cdots A_{r-1}), (A_1 A_2 \cdots A_{r-1})^{-1} (A_2 A_4 \cdots A_{r-2})] \quad \text{if } n=2.
\]
Here the matrices in $\SL(2, \C)$ is identified with points on sphere via the equivalence of quadratic forms $\det = x_1x_4 - x_2x_3$ and the Euclidean quadratic form $q = x_1^2 + x_2^2 + x_3^2 + x_4^2$ over $\C$.

\subsection{Automorphisms of $S(4, r)$}

Let $g > 1, n= 1, 2$ and $r = 2g + n$. Denote the Humphries generators of $\Mod(\Sigma_{g, n})$ by hyperelliptic $\alpha_1, \ldots, \alpha_{r-1}$ and non-hyperelliptic $\gamma$. The Dehn twist along $\gamma$ gives a nontrivial automorphism of $X(\Sigma)$, not lying in the image of the embedding $B_r \hookrightarrow \Mod(\Sigma)$. 

\begin{remark}
    $V(4, r)$, the space of unipotent upper triangular matrices whose symmetrization is of $\rk \leq 4$ is isomorphic to $S(4)^r \sslash \mathrm{O}(4)$. There is a $\deg 2$ morphism $S(4, r) \rightarrow V(4, r)$ and one can define the Coxeter invariant as
    \[
    c : V(4, r) \rightarrow \mathrm{Pin}(4) \sslash \mathrm{O}(4).
    \]
    The Dehn twist along $\gamma$ does not preserve the Coxeter invariant defined above. In particular, the mapping class group action does not descend to the action on $V(4, r)$.
\end{remark}

\subsection{Construction of $\Mod(\Sigma_{g, 1})$ and $\Mod(\Sigma_{g, 2})$}
Let
\[
[\rho] = [(A_1, \ldots, A_{r-1})] = [(A_i)]_{i=1}^{r-1}
\]
be a local system where each $A_i$ is the monodromy along the hyperelliptic generator of $\Sigma_{g, n}$ where $n = 1, 2$ and $r = 2g + n$. 

The $H^1(\Sigma, \Z / 2\Z)$-action is described by
\[
(a_i) \cdot [(A_i)] = [(a_i A_i)]
\]
where $(a_i) = (a_1, \ldots, a_{r-1}) \in (\Z / 2\Z)^{r-1} = H^1(\Sigma, \Z / 2\Z)$.

If $r = 1$, the boundary monodromy is always invariant by the action. 

If $r = 2$, the boundary monodromy is preserved if and only if 
\[
a_1a_3 \cdots a_{r-1} = 1 = a_2a_4 \cdots a_{r-2}.
\]

Note $Z(\SO(4)) = \Z / 2\Z$. $(b_i)_{i=1}^r \in (\Z / 2\Z)^r$ acts on $S(4, r)$ coordinatewise,
\[
(b_i) \cdot [(u_i)] = [(b_i u_i)]
\]
with $Z(\SO(4)) = \{(\pm1)_{i=1}^r \} \leq (\Z / 2\Z)^r$ acting trivially. By unraveling the correspondence, one has an isomorphism
\[
H^1(\Sigma, \Z / 2\Z) \rightarrow Z(\SO(4))^r \slash Z(\SO(4))
\]
defined by
\[
(a_1, \ldots, a_{r-1}) \mapsto [(a_1\cdots a_{r-1}, a_2 \cdots a_{r-1}, \ldots, a_{r-1}, 1)].
\]

Summarizing above computation, we obtain the following results.

\begin{proposition}
    The isomorphism $X(\Sigma_{g, n}) \simeq S(4, r)$ is equivariant under $(\Z / 2\Z)^{r-1}$-action defined above. If $n = 1$, the action preserves the Coxeter invariant. If $n = 2$, the action preserves the Coxeter invariant if and only if $b_1 b_r = 1$ where $(b_i)_{i=1}^r \in Z(\SO(4))^r$. 
\end{proposition}

Let $\Aut(S(4,r))$ be the relative automorphism group of $S(4, r)$, consisting of automorphisms of the complex variety $S(4, r)$ that preserve the Coxeter invariant morphism. Note if $n = 1$
\[
Z(\SO(4))^r / Z(\SO(4)) \simeq Z(SO(4))^{r-1} \hookrightarrow \Aut(S(4,r))
\]
and if $n = 2$
\[
Z(\SO(4))^{r-1} / Z(\SO(4)) = Z(\SO(4))^{r-2} \hookrightarrow \Aut(S(4,r)). 
\]

\begin{corollary}\label{mcg}
    Let $(g, n) \neq (1, 2)$. Then $\Mod(\Sigma_{g, 1})$ is isomorphic to
    \[
    \Aut(S(4,r)) / Z(\SO(4))^{r-1}
    \]
    and $\Mod(\Sigma_{g, 2})$ is isomorphic to
    \[
    \Aut(S(4,r)) / Z(\SO(4))^{r-2}.
    \]
    In particular, they can be defined without an a priori choice of a surface.
\end{corollary}

\end{document}